\documentclass[12pt]{amsart}
\usepackage{graphics,wrapfig, graphicx, mathrsfs}
\newtheorem{lemma}{Lemma}[section]
\newtheorem{theorem}{Theorem}[section]

\newtheorem{remark}{Remark}[section]

\numberwithin{equation}{section} \numberwithin{theorem}{section}
\numberwithin{example}{section} \numberwithin{remark}{section}
\numberwithin{figure}{section} \numberwithin{algorithm}{section}

\def\ba{\begin{array}}
\def\ea{\end{array}}
\def\bma{\left(\begin{matrix}}
\def\ema{\end{matrix}\right)}
\def\be{\begin{equation}}
\def\ee{\end{equation}}

\def\dfrac{\displaystyle\frac}

\def\mcF{\mathcal{F}}

\def\mfa{\mathfrak{a}}
\def\mfh{\mathfrak{h}}
\def\mfu{\mathfrak{u}}
\def\mfv{\mathfrak{v}}

\begin{document}

\title{Formation of singularity for the rotating shallow water system}
\author{Yupei Huang}
\address{Department of mathematics, Duke University}
\email{yh298@duke.edu}
\author{Chunjing Xie}
\address{School of mathematical Sciences, Institute of Natural Sciences,
Ministry of Education Key Laboratory of Scientific and Engineering Computing,
and SHL-MAC, Shanghai Jiao Tong Vniversity, 800 Dongchuan Road, Shanghai, China}
\email{cjxie@sjtu.edu.cn}

\date\today
\begin{abstract}
In this paper, we investigate the formation of singularity for general two dimensional and radially symmetric solutions for rotating shallow water system from different aspects. First, the formation of singularity is proved via the study for  the associated moments for two dimensional solutions. For the radial symmetric solutions, the formation of singularity is established for the initial data with compact support. Finally, the global existence or formation of singularity  for the radial symmetric solutions of the rotating shallow water system are analyzed in detail when the solutions are of the form with separated variables.
\end{abstract}

\maketitle
\section{Introduction and main results}

The rotating shallow water system
\begin{equation}\label{bheq}
\left\{
\begin{aligned}
&\frac{\partial h}{\partial  t}+\frac{\partial (hu)}{\partial x}+\frac{\partial (hv)}{\partial y}=0,\\& \frac{\partial u}{\partial t}+u\frac{\partial u}{\partial x}+v\frac{\partial u}{\partial y}-v=-\frac{\partial h}{\partial x},\\& \frac{\partial v}{\partial t}+u\frac{\partial v}{\partial x}+v\frac{\partial v}{\partial y}+u=-\frac{\partial h}{\partial y},
\end{aligned}
\right.
\end{equation}
is a widely adopted 2D model that describes the behavior of fluid in the regime of large scale geophysical fluid motion under the action of  Coriolis force (\cite{Gill1982, Majda, Pedlosky}). It can also be regarded as
an important extension of the compressible Euler equations with
additional force.  In the system \eqref{bheq},  $h$ denotes the height of the fluid, $u$ and $v$ are the  velocity in $x$ and $y$ directions,  respectively. For mathematical convenience, all physical parameters are scaled to
the unit (cf.\cite{Majda} for detailed discussion on scaling).
The first equation in \eqref{bheq} describes the conservation of the mass, while the second and third equations in \eqref{bheq} result from the conservation of momentum in $x$ and $y$ directions.

There are quite a few studies on the Cauchy problem for \eqref{bheq}, i.e., the system \eqref{bheq} supplemented with the initial conditions:
\begin{equation*}
h(0,x,y)=h_{0}(x,y),\quad
u(0,x,y)=u_{0}(x,y),\quad
v(0,x,y)=v_{0}(x,y).
\end{equation*}
 In \cite{ChTa:SIAM},   the prolonged existence of classic solution was obtained when the rotation is very fast. It is observed in \cite{ChengXie} that  the relative vorticity vanishes all the time if it vanishes initially. Furthermore,  the system \eqref{bheq} with zero relative vorticity  can be written as a Klein-Gordon system with quadratic nonlinearity so that global small solutions can be established with the aid of the method developed in \cite{Klainerman, Shatah}. Global small solutions were also established for one dimensional rotating shallow water system in \cite{ChengXie:SIAM} based on the study for one dimensional Klein-Gordon equation in \cite{Delort:1D}.  Furthermore, the lifespan of two dimensional classical solutions was also obtained in \cite{ChengXie} when the relative vorticity is sufficiently small at the initial time.  For more studies on the effect of rotation, see \cite{LiTa:rotation} and references therein.  A natural question is to know whether the solutions form singularity in finite time if the initial vorticity is not zero.

Note that the one dimensional rotating shallow water system is a typical example for the quasilinear hyperbolic system \cite{Dafermos}.  Inspired by  \cite{Lax, TW, CPZ}, it was showed in \cite{ChengXie:SIAM} that one dimensional rotating shallow system can form singularity in finite time.
For the multidimensional compressible Euler system,
Sideris first gave some sufficient condition on the formation of singularity for the three dimensional compressible Euler system\cite{Sideris:3D:singularity}. This approach was generalized in various setting, see \cite{Rammaha} for two dimensional Euler system, \cite{SiderisWang, HLiWang} for the compressible Euler system with damping, etc. It is natural to know whether one can also prove the similar results as that for the Euler system.  For more study on the rotating shallow water system from various aspects, one may refer to \cite{Babin, Beale, Bresch,ChengMa:zonal, Li} and references therein.

Adapting the method developed in \cite{Sideris:3D:singularity}, we first study the formation of singularity for two dimensional rotating shallow water system.  Although we can't get a Riccati equation for the radial component of the moment (as that  defined in \cite{Sideris:3D:singularity}) to prove the formation of singularity, one of the key observations in this paper is  that the radial component of moment and the angular component of the moment forms a pair of quantities that obey some conservation laws.
Hence the formation of singularity for the rotating shallow water system is a consequence of concentration of energy under some condition of the initial profile.

The first result in this paper can be stated as follows.
\begin{theorem}\label{thm2D1}
	Suppose there exists an $\bar{h}$, such that $(h_0-\bar{h},u_0,v_0)$ have compact support in $B_{R}=\{(x,y): x^2+y^2\leq R^2\}$.  If  $m(0)>0$ and
	  \begin{equation}
P_1(0)^2+(E(0)+P_2(0))^2\geq\pi(R+2\pi\sqrt{\bar h} )^{4}E(0)\|{h_{0}}\|_{L^\infty},
 \end{equation}
  where
\begin{equation*}
P_1(t)=\int_{\mathbf{R}^2}hux+hvy dxdy, \quad  P_2(t)=\int_{\mathbf{R}^2}hvx-huydxdy,
\end{equation*}
and
\begin{equation*}
E(t)=\int_{\mathbf{R}^2} h(u^{2}+v^{2})+h^{2}-\bar{h}^2dxdy,\quad m(t)=\int_{\mathbf{R}^2} h-\bar{h} dxdy,
\end{equation*}
then the classical solutions of \eqref{bheq} lose $C^1$ regularity before $2\pi$.
\end{theorem}

There are a few remarks in order.
\begin{remark}
The quantity $P_1$ is what Sideris called ``radial component of moment", while the quantity $P_2$ can be regarded as  the ``angular component of moment".
 This result as well as  \cite[Theorem 1]{Sideris:3D:singularity} is a consequence of the finite speed of propagation of the system and the situation that the majority of the mass is moving outside while the finite speed of propagation automatically sets up a barrier outside of the wave front. Therefore, as long as the outward velocity in the interior is large enough, formation of singularity is inevitable, which results in breakdown of $C^1$ solution.
\end{remark}
\begin{remark}\label{rmk1}
		Suppose $ m(0)\geq 0$, while $P_{1}(0)^2+(E(0)+P_{2}(0))^2\leq(R+2\pi\sqrt{\bar h} )^{4}E(0)\|h_0\|_{L^\infty}$. Define
\begin{equation}\label{scale1}
	h_{0}^{(\lambda)}(x,y)=\frac{1}{\lambda}h_{0}(x,y),\quad  u_{0}^{(\lambda)}(x,y)=\lambda u_{0}(x,y),\quad v_{0}^{(\lambda)}(x,y)=\lambda v_{0}(x,y).
	\end{equation}
If $\lambda$ is sufficiently large, there exists a  constant $C$ such that
\begin{equation*}
	C^{-1}\leq P^{(\lambda)}_{1}(0), \,
	P^{(\lambda)}_{2}(0), \,
	\frac{E^{(\lambda)}(0)}{\lambda}, \,
	R+2\pi \sqrt{\frac{\bar h}{\lambda}} \leq C,
	\end{equation*}
	where $P^{(\lambda)}_1$, $P^{(\lambda)}_2$, $E^{(\lambda)}$ are the associated quantities for the scaled initial data.
Therefore, one has
	 \begin{equation*}
	 P^{(\lambda)}_{1}(0)^2+(E^{(\lambda)}(0)+P^{(\lambda)}_{2}(0))^2\geq\pi\left(R+\sqrt{\frac{\bar h}{\lambda}}\right)^{4}E^{(\lambda)}(0)\|{h^{(\lambda)}_{0}}\|_{\infty}.
	 \end{equation*}The above argument shows that if the initial profile has large concentration of energy, the system is likely to blow up in finite time.
	\end{remark}
	\begin{remark}
The condition $m(0)>0$ plays a crucial role in Theorem \ref{thm2D1}. Indeed,   $(h_0, u_0, v_0)= ( \bar h, 0, 0)$, with some constant $\bar h$ clearly satisfies
		\begin{equation}
	m(0)=0 \quad \text{and}\quad P_{1}(0)^2+(E(0)+P_{2}(0))^2=0\geq 0=\pi(R+2\pi\sqrt{\bar h} )^{4}E(0)\Arrowvert{h_{0}}\Arrowvert_{L^\infty}.
	\end{equation}
	However, there is a unique global classical solution $(h, u, v)= ( \bar h, 0, 0)$.  Moreover, we can find \begin{equation*}
		\left\{
	\begin{aligned}
	&m(0)=0,\\
	&P_{1}(0)^2+(E(0)+P_{2}(0))^2\geq\pi(R+2\pi\sqrt{\bar h} )^{4}E(0)\Arrowvert{h_{0}}\Arrowvert_{L^\infty}.
	\end{aligned}
	\right.
	\end{equation*} only if  there exists a constant $\bar h$ such that $h_0\equiv \bar h$,  $u\equiv 0$,  $v\equiv 0$.
\end{remark}

\begin{remark}
	Note the system \eqref{bheq} enjoys a scaling symmetry, i.e., $(h_\lambda, u_\lambda, v_\lambda)$ defined by 
	 \begin{equation}\label{defhlambda}
	h_{\lambda}(x,y,t)=\lambda^{2}h\left(\frac{x}{\lambda},\frac{y}{\lambda},t\right),
	\,\, u_{\lambda}(x,y,t)=\lambda u\left(\frac{x}{\lambda},\frac{y}{\lambda},t\right),\,\, v_{\lambda}(x,y,t)=\lambda v\left(\frac{x}{\lambda},\frac{y}{\lambda},t\right)
	\end{equation}
is a solution of \eqref{bheq} so long as $(h,u,v)(x,y, t)$ is a solution of \eqref{bheq}. This, together with the transformation in Remark \ref{rmk1}, shows that there exists a solution with small initial data (of $C^{0}$ norm), which forms singularity in finite time. For $\lambda \geq 1$, under the transformation of \eqref{defhlambda}, the $C^1$ norm does not change much and the configuration of the initial data tends to be localized.
\end{remark}

A particular interesting case is to study  the radial symmetric solutions of the rotating shallow water system.
Suppose that $(h,u,v)$ is of the form
  \begin{displaymath}
 h(x,y, t)=h(r,t),\quad u(x,y,t)=\frac{U(r,t)x}{r}-\frac{V(r,t)y}{r}\quad \text{and}\quad v(x,y,t)=\frac{V(r,t)x}{r}+\frac{U(r,t)y}{r},
\end{displaymath}
where $r=\sqrt{x^2+y^2}$ and $U$ and $V$ are radial and angular velocity.
Then $(h,  U, V)$ satisfies
\begin{equation}\label{rotation}
\left\{
\begin{aligned}
&\frac{\partial h}{\partial t}+\frac{\partial (hU) }{\partial r}+\frac{hU}{r}=0,\\
&
\frac{\partial U}{\partial t}+U\frac{\partial U}{\partial r}+\frac{\partial h}{\partial r}-V-\frac{V^2}{r}=0,\\
&
\frac{\partial V}{\partial t}+U\frac{\partial V}{\partial r}+U+\frac{VU}{r}=0.
\end{aligned}
\right.
\end{equation}

The system \eqref{rotation} is a typical quasilinear hyperbolic system with source terms which could be singular near the origin. Different from the methods  in \cite{Lax, John} which heavily rely on the analysis along the characteristics,
Sideris (\cite{Sideris:hyp:singularity})  proved the formation of singularity for hyperbolic system  via the energy method. Inspired by \cite{Sideris:hyp:singularity}, we have the following results on the formation of singularity for the Cauchy problem for \eqref{rotation}, i.e., the system \eqref{rotation} together with the initial data
\begin{equation}\label{rotationIC}
(h, U, V)(0, r)=(h_0, U_0, V_0)(r).
\end{equation}

\begin{theorem}\label{mainthm2}
Assume that there exists an $ \bar{h}>0$ such that $(h_0-\bar{h},U_0, V_0)\in C_c^1((0, \infty))$ and denote $\text{supp} (h_0-\bar h, U_0, V_0)\subset [\underline{\beta}, \bar \beta]$ with $\underline{\beta}$, $\bar\beta>0$. Suppose that
\begin{equation}\label{assu0}
h_0(r) \geq \underline{h}>0 \quad \text{and}\quad   U_0(r) \leq 0 \text{ for all } r\in [\underline{\beta}, \bar \beta],
\end{equation}   and that
\begin{equation}\label{assul1}
\|U_0\|_{L^\infty(0, \infty))}\leq A \|U_0\|_{L^1(0, \infty)) }
\end{equation}
for some positive  constant $A$ independent of $U_0$. If, in addition,
 \[
 \Arrowvert U_0 \Arrowvert_{L^\infty}\geq C\]
  for some constant $C$ depending on $A$, $\underline{h}$, $\|V_0\|_{L^\infty}$, $\underline{\beta}$ and $\bar\beta$,   the classical solutions of the problem  \eqref{rotation}-\eqref{rotationIC} must form singularity in finite time.
\end{theorem}

There are several remarks on Theorem \ref{mainthm2}.
\begin{remark}
In fact, the proof of Theorem \ref{mainthm2} in Section \ref{sec3} also gives the upper bound of the lifespan of the classical solutions. We can also use Theorem \ref{thm2D1} to prove the formation of singularity for some cases which satisfy the assumptions of  both Theorem \ref{mainthm2} and Theorem \ref{thm2D1}. The upper bound of lifespan obtained by Theorem \ref{thm2D1} is $2\pi$, while the upper bound of the lifespan obtained by Theorem \ref{mainthm2} is much smaller. 
\end{remark}
\begin{remark}
The assumption on negativity of $U_0$ in \eqref{assu0} is used to guarantee $\alpha$ defined in \eqref{defalpha} to be positive, so this condition is not an optimal condition. Theorem \ref{mainthm2} is proved by the energy method and $L^\infty$ bound for $U_0$ is controlled  by the integral at the initial data (cf. \eqref{assumption1}). Of course, it is easy to see there are a large class of initial data satisfying \eqref{assu0}-\eqref{assul1}.
\end{remark}

\begin{remark}
 In order to get the formation of singularity for one dimensional quasilinear hyperbolic system, the sign of the derivative is crucial (\cite{Lax, John}).  This is the reason why we study the case where there is a negative bump for the radial velocity.
 \end{remark}

\begin{remark}
In Theorem \ref{mainthm2}, some portion of the mass is moving towards the origin with large velocity and the angular velocity is not large so that it doesn't prevent the particle paths from collision near the origin.  Thus the solution becomes singular in a  short period of time.
\end{remark}

Note that the system \eqref{rotation} away from the origin is  a typical one dimensional hyperbolic system. It is interesting to study the solutions of \eqref{rotation} with the form of separated variables. In fact, there are some studies on this special kind of solutions for rotating shallow water system or compressible Euler system. A special class of these kind of solutions for rotating shallow water system has been studied in \cite{CHESN}. The formation of singularity  for the solutions with separated variables  for the compressible Euler system was investigated in \cite{LiWang}.

The solutions under the form of separated variables can be written as
\begin{equation}\label{separation}
 (h,U, V)=(\mfh(t) \mfa_1(r), \mfu(t) \mfa_2(r),\mfv(t) \mfa_3(r)).
 \end{equation}
 In fact, we have the following results for the rotating shallow water system.
\begin{theorem}\label{mainthm3}
If the solution of the system \eqref{rotation} has the form \eqref{separation}, then the solution is either time periodic or  blow up in a finite time.
\end{theorem}
The detailed presentation for Theorem \ref{mainthm3} is stated in Theorem \ref{mainthm3r}.

The rest of the paper is organized as follows. Theorem \ref{thm2D1} is proved in Section \ref{sec2}.  We give the proof of Theorem \ref{mainthm2} in Section \ref{sec3}. The detailed formulation for the governing equations of  the solutions under the form \eqref{separation} for the system \eqref{rotation} is presented in Section \ref{sec4}. Furthermore, the global existence of time periodic and the blowup  of the solutions are studied in detail in Section \ref{sec4}.

\section{Formation of singularity for two dimensional rotating shallow water system}\label{sec2}

The main objective of this section is to prove Theorem \ref{thm2D1}.
\begin{proof}[Proof of Theorem \ref{thm2D1}]
Integration by parts yields
	\begin{equation}\label{conservation of momentum}
	\left\{
	\begin{aligned}
	&P_{1}'(t)=E(t)+P_{2}(t),\\
	&P_{2}'(t)=-P_{1}(t).
	\end{aligned}
	\right.
	\end{equation}
Similarly, one has
	\begin{equation}\label{conservation of mass}
E(t)= E(0) \quad\text{and}\quad  m(t)=m(0).
\end{equation}
It follows from \eqref{conservation of momentum} and  \eqref{conservation of mass} that there exists $p$ and $q$ such that
	\begin{equation}
	\left\{
	\begin{aligned}
	&P_{1}(t)=p\sin(t)+q\cos(t),\\
	&P_{2}(t)=- E(0)+p\cos(t)-q\sin(t),
		\end{aligned}
	\right.
	\end{equation}
where $q=P_{1}(0)$ and $p=-P_{2}(0)+ E(0)$.

Since $(h_0-\bar h, u, v)$ has the support in $B_R$, the straightforward computations for the characteristic speed for \eqref{bheq} show that the support of $(h-\bar h, u, v)$ must be contained in $S(t)=\{(x,y)\leq R+\sigma t\}$ with $\sigma =\sqrt{\bar h}$. Hence  there exists a $T\in[0,2\pi]$ such that
	\begin{equation}\label{eq2.4}
	\begin{split}
	P_{1}(0)^2+( E(0)+P_{2}(0))^2&=p^2+q^2=P_{1}^2(T)=\left(\int_{S(T)} hux+hvy dxdy\right)^2\\&\leq \left(\int_{S(T)} \vert hux+hvy \vert  dxdy\right)^2=J_1(T).
	\end{split}
	\end{equation}
Applying Cauchy-Schwarz inequality and H\"{o}lder inequality yields
	\begin{equation}\label{eq2.5}
	\begin{split}
	J_1(t)&=\left(\int_{S(t)} \vert hux+hvy \vert dxdy\right)^2\\&\leq\left(\int_{S(t)}h(u^2+v^2)^{\frac{1}{2}}(x^2+y^2)^{\frac{1}{2}}dxdy\right)^2\\&\leq\int_{S(t)}h(u^2+v^2)dxdy\int_{S(t)}h(x^2+y^2)dxdy=J_2(t)J_3(t),
	\end{split}
	\end{equation}
where
	\begin{equation*}
	J_2(t)=\int_{S(t)}h(u^2+v^2)dxdy=E(t)+\int_{S(t)} (\bar{h}^2-h^2))dxdy
	\end{equation*}
	and
	\begin{equation*}
	\begin{split}
	J_3(t)&=\int_{S(t)}h(x^2+y^2)dxdy.
	\end{split}
	\end{equation*}	
Using H\"{o}lder inequality yields
	\begin{equation}\label{eq2.6}
	\begin{aligned}
	\int_{S(t)} h^2dxdy\geq  &\frac{\left(\int_{S(t)}hdxdy\right)^2}{\vert{S(t)}\vert}=\frac{\left(\int_{S(t)}\bar{h}dxdy+m(0)\right)^2}{\vert{S(t)}\vert}\\
	> & \bar{h}^2 |S(t)|=\int_{S(t)} \bar{h}^2 dxdy,
	\end{aligned}
	\end{equation}
	where $\vert S(t)\vert$ denotes the area of $S(t)$ and $m(0)>0$ has been used. Thus one has
	\begin{equation*}
	J_2(T)=E(T)+\int_{S(T)} (\bar{h}^2-h^2)dxdy< E(T)=E(0),
	\end{equation*}
	where \eqref{eq2.6} is used to get the strict inequality.
Furthermore, for $T\in [0, 2\pi]$, one has
	\begin{equation*}
	\begin{split}
	J_3(T)&=\int_{S(T)}h(x^2+y^2)dxdy\leq(R+\sigma T)^2\int_{S(T)}hdxdy=(R+\sigma T)^2\int_{S(T)}h_{0}dxdy\\
	&\leq\pi(R+2\pi\sigma )^{4}\|{h_{0}}\|_{L^\infty}.
	\end{split}
	\end{equation*}
Therefore, these two estimates together with \eqref{eq2.4} and \eqref{eq2.5}  yield
	\[
	P_{1}(0)^2+(E(0)+P_{2}(0))^2<\pi(R+2\pi\sigma )^{4}E(0)\|{h_{0}}\|_{L^\infty}.
	\]
	 This leads to a contraction so that the proof of the theorem is completed.
	\end{proof}

\section{Formation of singularity for radially symmetric solutions}\label{sec3}

The objective of this section is to prove Theorem \ref{mainthm2}.

\begin{proof}[Proof of Theorem \ref{mainthm2}] The proof is divided into 4 steps.

{\it Step 1. Preliminaries.} The straightforward calculations show that the characteristic speeds of the system \eqref{rotation} are
\[
\lambda_1= U-\sqrt{h}, \quad \lambda_2=U, \quad \lambda_3=U+\sqrt{h}.
\]
Hence if the solutions have compact support, then the finite propagation speed of the support is $\sigma=\bar{h}^{\frac{1}{2}}$.
Without loss of generality, we assume that  the support of $(h_0-\bar{h},U_0,V_0)$ is contained in the interval $A_0$ defined as
 \[
 A_0=\left\{r: \frac{1}{2} \leq r\leq 1\right\}.
\]
The proof can be proceeded with minor modifications for the case where $\text{supp}(h_0-\bar{h}, U_0, V_0)$ is contained in $[\underline{\beta}, \bar{\beta}]$ with general $\underline{\beta}$, $\bar\beta>0$. Let $\epsilon>0$ be a small positive constant and
denote
\begin{equation}\label{defnu}
A(t)=\left\{r\mid r\in\left[\frac{1}{2}-\sigma t,1+\sigma t\right]\right\},\quad \nu = \Arrowvert U_0 \Arrowvert_{L^\infty}^{\frac{1}{3}-\epsilon},\quad \text{and}\quad
\bar T= \dfrac{2}{\nu}.
\end{equation}
Assume that  $\Arrowvert U_0 \Arrowvert_{L^\infty}$ is large enough so that
\[
A(\bar{T}) \subseteq \{r\mid \frac{1}{3}\leq r \leq \frac{4}{3} \}.
\]
The main goal is to prove that the system \eqref{rotation} forms singularity before $t=\bar{T}$.

{\it Step 2. Energy estimate.}
Assume that the system \eqref{rotation} has a classical solution in $ [0,\bar{T}]\times \mathbb{R}$.
Then  for any fixed $T\in [0,\bar{T}]$, let $w(r, t; T)$ be defined as follows
		\begin{equation*}
	w(r,t; T)=\left\{
	\begin{aligned}
	&0, \quad &t>T,\\
	&e^{-r}(T-t)^2\mu(r), \quad &t\in[0,T],
	\end{aligned}
	\right.
	\end{equation*}
	where $\mu(r) \in C_c^\infty (0, \infty)$ satisfies $\mu(r)\equiv 1$ for  $ r\in (\frac{1}{3}, \frac{4}{3})$ and $\mu \in [0,1]$ for $r\in (0, \infty)$.
	Multiplying the second equation in \eqref{rotation} with $w(t)$ yields	
	\begin{equation}\label{conservation of radial momentum with weight}
	\begin{aligned}
	&- \int_{0}^{T}\int_{A(t)}2(T-t)e^{-r}h {U} \mu(r)\ drdt =- \int_{A(0)}h_0{U_0}T^2e^{-r} \mu(r)dr- I(T)\\
	&\quad +
	\int_{0}^{T}\int_{A(t)}(T-t)^2e^{-r}\left(h {U}^2+\frac{h{U}^2}{r}+\frac{h^2}{2}\right) \mu(r)dr dt,
	\end{aligned}
	\end{equation}
		where $I(t)=\sum_{i=1}^3 I_i(t)$ with
		\begin{equation*}
	I_1(T)=\int_{0}^{T}\int_{A(t)}\frac{hV^2e^{-r}}{r}(T-t)^2 \mu(r) drdt,
	\quad I_2(T)=\int_{0}^{T}\int_{A(t)}hVe^{-r}(T-t)^2 \mu(r) drdt,
	\end{equation*}
	and
	\begin{equation}\label{defalpha}
I_3(T)=\int_{0}^{T}\int_{\mathbb{R}^{+}\setminus A(t)}\left(\frac{h^2}{2} +hU^2\right)e^{-r}(T-t)^2\mu_rdrdt.		
	\end{equation}	
	
	Let
\[
\alpha=-\int_{A(0)}h_0 {U_0}e^{-r}\mu(r)dr\quad \text{and}\quad
	F(T)=-\int_{0}^{T}\int_{A(t)}h {U}e^{-r}(T-t)^2\mu(r)drdt.
\]
According to \eqref{conservation of radial momentum with weight} and Young's inequality, one has
	\begin{equation}
	\begin{split}
	F'(T)&=\alpha T^2-I(T)+\int_{0}^{T}\int_{A(t)}\left(\frac{h {U}^2}{2}+\frac{h {U}^2}{r}+{h^2}\right)(T-t)^2 \mu(r)drdt,\\
	&\geq\int_{0}^{T}\int_{A(t)} \frac{3}{2}(h{U})^{4/3}(T-t)^2 e^{-r} \mu(r)drdt+\alpha T^2-I(T).
	\end{split}
	\end{equation}

	We claim that if $\Arrowvert U_0 \Arrowvert_{L^\infty}$ is sufficiently large, then
	\begin{equation}\label{claim1}
I(T)\leq\frac{4\alpha T^2}{5}\quad \text{for all } T\in [0,\bar{T}].
	\end{equation}
		The detailed  proof for \eqref{claim1} is given in Step 4.
		
{\it Step 3. Derivation of Riccati type inequality.} Assume that \eqref{claim1} is true and we continue the proof for the theorem.		
	Define	
	\[
	\phi(t)=\int_{A(t)}e^{-r}dr
	\quad \text{and}\quad \varphi(T)=\int_{0}^{T}(T-t)^2\phi(t)dt.
	\]
It follows from
from \eqref{conservation of radial momentum with weight}, \eqref{claim1}, and Jensen's inequality that one has
\begin{equation}\label{time intrgral of weighted moment}
F'(T)\geq\frac{\alpha T^2}{5}+\int_{0}^{T}(T-t)^2\int_{A(t)}\frac{3}{2}(h {U})^{4/3}e^{-r} \mu(r) drdt\geq\frac{\alpha T^2}{5}+\frac{3}{2}\varphi(T)\left(\frac{F(T)}{\varphi(T)}\right)^{4/3}.
\end{equation}

The straightforward computations give
\begin{equation}\label{defa}
\frac{\phi'(t)}{\phi(t)}=\sigma\dfrac{1+e^{-2\sigma t-0.5 }}{1-e^{-2\sigma t-0.5}}\leq \frac{\phi'(0)}{\phi(0)}= \sigma \frac{1+e^{-1/2}}{1-e^{-1/2}}:=a.
\end{equation}
Note that both $\phi$ and $\phi'$ are positive, thus  one has
\begin{equation}\label{middle step1}
\begin{split}
a\varphi(T)&\geq\int_{0}^{T}(T-t)^2\phi'(t)dt=-\phi(0)T^2+2\int_{0}^{T}(T-t)\phi(t)dt\\&=-\phi(0)T^2+\varphi'(T).
\end{split}
\end{equation}
Furthermore, according to L'Hopital rule,
\begin{equation}\label{assumption2}
\lim_{T\rightarrow 0}\frac{F(T)}{\varphi(T)}=\lim_{T\rightarrow 0}\frac{F'(T)}{\varphi'(T)}=\lim_{T\to 0}\frac{F'''(T)}{\varphi'''(T)} =\lim_{T\to 0} \frac{-\int_{A(T)} (hU)(r, T)e^{-r} \mu(r)dr}{\phi(T)}=\frac{\alpha}{\phi(0)}.
\end{equation}
If $\Arrowvert U_0 \Arrowvert_{\infty}$ is sufficiently large, then
\begin{equation}\label{assumption1}
\frac{3}{2}\left(\frac{\alpha}{5\phi(0)}\right)^{1/3}\geq a+3\nu,
\end{equation}
where $a$ is defined in \eqref{defa}.
For any $\zeta\in [\frac{8}{27} (a+3\nu)^3, \frac{\alpha}{5\phi(0)}]$,  there exists a $\tilde{T}>0$ such that  $F(t)>\zeta \varphi(t)$ for any $ t\in[0,\tilde{T}]$.

Next, we claim that for any $\zeta\in [\frac{8}{27} (a+3\nu)^3, \frac{\alpha}{5\phi(0)}]$,
\begin{equation}\label{estF}
F(T)\geq\zeta\varphi(T) \quad \text{for any } T \in [0,\bar{T}].
\end{equation}
Indeed, if $\zeta\in [\frac{8}{27} (a+3\nu)^3, \frac{\alpha}{5\phi(0)}]$, for any $T\in [0, \tilde{T}]$, then it follows from \eqref{time intrgral of weighted moment} and \eqref{assumption2}  that
\begin{equation}\label{assumption3}
\begin{split}
F'(T)\geq\frac{\alpha T^2}{5}+\varphi(T)\frac{3}{2}\zeta^{4/3}>\zeta \phi(0)T^2+\zeta a\varphi(T)\geq\zeta \varphi'(T).
\end{split}
\end{equation}
Denote
\[
\mcF(T)=\frac{F(T)}{\varphi(T)}.
\]
The straightforward computations give
\begin{equation}\label{eq3.11}
\begin{aligned}
\mcF^{'}( T)= &\frac{F^{'}(T)}{\varphi(T)}-\frac{\varphi^{'}(T)}{\varphi(T)}\mcF(T)
\geq \frac{\alpha T^2}{5\varphi(T)}-\frac{\varphi'(T)}{\varphi(T)} \mcF(T)+\frac{3}{2}(\mcF(T))^{4/3},
\end{aligned}
\end{equation}
where \eqref{time intrgral of weighted moment} is used.
If $\mcF$ achieves $\zeta$ at $\hat{T}\in [0, \bar{T}]$ for the first time, i.e.,
\[
\hat{T}=\inf\{ T\in [0, \bar{T}]: \mcF(T)=\zeta\}.
\]
It follows from  \eqref{middle step1} that one has
\begin{equation}\label{key estimate1}
\begin{split}
 \mcF'(\hat T) \geq\frac{a\mcF(\hat T)\varphi(\hat T)-\varphi'(\hat T)\mcF(\hat T)+\dfrac{\alpha \hat T^2}{5}}{\varphi(\hat T)}\geq \frac{\phi(0)\hat T^2 (\dfrac{\alpha}{5\phi(0)}-\zeta)}{\varphi(T)}>0.
\end{split}
\end{equation}
Thus, for any $\zeta\in [\frac{8}{27} (a+3\nu)^3, \frac{\alpha}{5\phi(0)}]$, the claim \eqref{estF} always holds, i.e.,
\begin{equation*}
	\mcF(T)\geq \zeta \quad \text{for any } T\in[0,\bar{T}].
\end{equation*}
In particular, one has
\begin{equation}\label{key estimate2}
\mcF(T)\geq\frac{\alpha}{5\phi(0)}\quad \text{for any } T\in[0,\bar{T}].
\end{equation}

Notice $\phi(t)$ is monotonically increasing, thus by \eqref{middle step1}, {for any } $T\geq\frac{\bar{T}}{2}$, one has
\begin{equation}\label{middle step2}
 \frac{\varphi'(T)}{\varphi(T)}\leq a+\dfrac{\phi(0)T^2}{\int_{0}^{T}\phi(t)(T-t)^2dt}\leq a+\frac{T^2}{\int_0^T (T-t)^2dt}\leq a+\frac{3}{T}  < a+3\nu .
\end{equation}
Combining with \eqref{eq3.11} gives
\begin{equation}\label{key estimate3}
\mcF'(T)\geq \frac{3}{2}\mcF^{4/3}(T)-(a+3\nu )\mcF(T)>0 \quad \text{for any }T\geq \frac{\bar{T}}{2}.
\end{equation}
This, together with \eqref{key estimate2}, yields
\begin{equation}\label{estimate}
\begin{aligned}
\frac{\bar{T}}{2}
\leq &
\int_{\mcF(\bar{T}/{2})}^{\mcF(\bar{T})}\frac{ds}{\frac{3}{2} s^{4/3}-(a+3\nu)s}
\leq& \int_{\frac{\alpha}{5\phi(0)}}^{+\infty}\frac{ds}{\frac{3}{2} s^{4/3}-(a+3\nu)s} \\
\leq &
\int_{\frac{\alpha}{5\phi(0)}}^{+\infty}\frac{ds}{s^{4/3}}\leq  C\left(\frac{\alpha}{5\phi(0)}\right)^{-\frac{1}{3}}
\end{aligned}
\end{equation}
provided  $\Arrowvert U_0 \Arrowvert_{L^{\infty}}$ is sufficiently large.
Therefore, we have
\begin{equation}
 	\Arrowvert U_0 \Arrowvert_{L^{\infty}}^{\epsilon-\frac{1}{3}}\leq C \frac{1}{2}\bar{T}\leq C\left(\frac{\alpha}{5\phi(0)}\right)^{-\frac{1}{3}}\leq C 	 \Arrowvert U_0\Arrowvert_{L^{\infty}}^{-\frac{1}{3}}.
 \end{equation}
This leads to a contradiction if $\Arrowvert U_0 \Arrowvert_{L^{\infty}}$ is sufficiently large. Hence the classical solution forms singularity before $t=\bar{T}$.
Therefore, it suffices to prove \eqref{claim1} in order to complete the proof of the theorem.

{\it Step 4. Proof of \eqref{claim1}.}
Note that $(h, U, V)(r, t)=(\bar h, 0, 0)$ for $r \in \mathbb{R}_+ \setminus A(t)$. Hence there exists a constant $C>0$ such that
\[ \frac{1}{C}\|U_0\|_{L^\infty}T^2 \leq \frac{4\alpha T^2}{5}\leq C \Arrowvert U_0 \Arrowvert_{L^\infty}T^2 \quad \text{and}\quad  I_3(T) \leq C T^3.
\]
 In order to complete the proof, it suffices to get the estimate of $I_1$ and $I_2$. This is established by combining the $L^\infty$ estimate for $V$ through Lagrangian form and the conservation of mass.
 Let
 \begin{equation}\label{eqparticlepath}
 \left\{
\begin{aligned}
&\frac{\partial \tilde{X}}{\partial t}=\tilde{U}(t, X_0),\\
& \tilde{X}(0, X_0)=X_0,
\end{aligned}
\right.
\end{equation}
where
\[
(\tilde{h},  \tilde{U}, \tilde{V}) (t,X_0)=(h, U, V)(t,\tilde{X}(t,X_0)).
\]
The Lagrangian form the system \eqref{rotation} is
\begin{equation}\label{lagrange}
\left\{
\begin{aligned}
&
\frac{\partial \tilde{X}}{\partial X_0}\left(\frac{\partial \tilde{h}}{\partial t} +\frac{\tilde{h}\tilde{U}}{\tilde{X}}\right)+\tilde{h}\frac{\partial \tilde{U}}{\partial X_0}=0,\\
&\frac{\partial \tilde{X}}{\partial X_0}\left(\frac{\partial \tilde{U}}{\partial t}-\frac{\tilde{V}^2}{\tilde{X}}-\tilde{V}\right)+\frac{\partial \tilde{h}}{\partial X_0}=0,
\\
& \frac{\partial \tilde{V}}{\partial t}+\frac{\tilde{V}\tilde{U}}{\tilde{X}}+\tilde{U}=0.
\end{aligned}
\right.
\end{equation}
The system \eqref{lagrange} has a $C^1$ solution  in
$[0,\bar{T}]\times \mathbb{R}_+$
 as long as the solutions of the system \eqref{rotation} belongs to  $C^1$ in
$[0,\bar{T}]\times \mathbb{R}_+$.

Firstly,  by the existence theorem of ODE theory, for any $r \in [0, \infty)$ and $t\in [0,\bar{T}]$, there exists an  $X_0$ such that $\tilde{X}(t, X_0)=r$. Moreover, note that  $A(\bar{T}) \subseteq \{r\mid \frac{1}{3}\leq r \leq \frac{4}{3} \}$, by the uniqueness theorem of ODE, one has
\begin{equation}\label{material}
    \left\{
    \begin{aligned}
    &\tilde{X}(T,X_0)=X_0 ,\,\,\tilde{V}(T,X_0)=0, \text{ for any }  X_0 \in \mathbb{R}^{+}\setminus \{r\mid \frac{1}{3}\leq r \leq   \frac{4}{3} \},\\
    & \frac{1}{3}\leq \tilde{X}(t,X_0) \leq \frac{4}{3},\,\,\text{for any } X_0\in \{r\mid \frac{1}{3}\leq r \leq \frac{4}{3} \}
     \end{aligned}
   \right.
\end{equation}
Hence one needs only to get the upper bound of $\tilde{V}(T,X_0)$ for $X_0\in \{r\mid \frac{1}{3}\leq r \leq \frac{4}{3} \}.$
It follows from \eqref{eqparticlepath} and the third equation of \eqref{lagrange} that one has
\[\frac{\partial(\tilde{X}\tilde{V}+\frac{\tilde{X}^2}{2})}{\partial t}=0.
\]
 Thus
\begin{equation}\label{rotation velocity}
    \tilde{V}=\frac{V_0(X_0)X_0+\frac{X_0^2}{2}-\frac{\tilde{X}^2}{2}}{\tilde{X}}.
\end{equation}
This, together with \eqref{material}, yields that  there exists a constant $C>0$ such that
 \[
 -C\leq\tilde{V}(T,X_0)\leq C  \text{ for any } X_0\in \{r\mid \frac{1}{3}\leq r \leq \frac{4}{3} \}.
 \]
Since the compact support of V is contained in $ A(t)\subset  \{r\mid \frac{1}{3}\leq r \leq \frac{4}{3} \}$, one has
\begin{equation}\label{upper bound for rotational velocity}
\Arrowvert V(T)\Arrowvert_{L^{\infty}} \leq C.
\end{equation}
Furthermore, it follows from \eqref{eqparticlepath} and the first equation in \eqref{lagrange} that one has
\[
\dfrac{\partial}{\partial t}\left(\tilde{h}\tilde{X} \frac{\partial \tilde{X}}{\partial X_0}\right)=0.
\]
This gives
\begin{equation}\label{height}
\tilde{h}(T,X_0)=\frac{h_0(X_0)X_0}{\tilde{X}\frac{\partial \tilde{X}}{\partial X_0}}.
\end{equation}
Differentiating \eqref{eqparticlepath} with respect to $X_0$ yields
that $\frac{\partial \tilde{X}}{\partial X_0}\geq 0$ for any $ t\in [0,\bar{T}]$.
This, together with \eqref{height}, shows that
	$h\geq 0$ for any $t\in [0,\bar{T}]$.
As was proved in  Theorem \ref{thm2D1}, the mass
\[
m(t)=\int_{\mathbb{R}^2} h-\bar{h} dxdy =\int_0^\infty (h(r, t)-\bar h) rdr =\int_{A(t)} (h-\bar h) rdr
\]
is conserved. Note that $A(t) \subset [1/3, 4/3]$. Hence there exists a constant $C>0$ such that
\begin{equation*}
I_1(T)\leq \sqrt{C} \int_{0}^{T}\int_{A(t)}\frac{h(T-t)^2}{r}  drdt\leq C \int_{0}^{T}\int_{A(t)} h(T-t)^2r drdt \leq C T^3.
\end{equation*}
Similarly, one has $I_2(T)\leq C T^3$ for some constant $C>0$.
Therefore,  if $\Arrowvert U_0 \Arrowvert_{L^\infty}$ is large enough, the estimate \eqref{claim1} must hold.

Hence the proof of the theorem is completed.
\end{proof}
\begin{remark}
In fact, we can replace $\bar T$ in \eqref{defnu} by $\bar{T}=\frac{N}{\Arrowvert U_0 \Arrowvert_{L^\infty}^{\frac{1}{3}}}$ with sufficiently large constant $N$. Indeed, we need only to replace  $a+3\mu$   in \eqref{assumption1} by  $a+\frac{6\mu}{N}$.   If $N$ is sufficiently large, one can also get contradiction for \eqref{estimate} and the proof can be proceeded with some minor modifications.
\end{remark}

\section{The solutions of separated variable form}\label{sec4}
In this section, we give the detailed description for Theorem \ref{mainthm3} and its proof.

 Note the system \eqref{rotation} enjoys the scaling symmetry, i.e.,  if $(h,U, V)$ is the solution of the original system \eqref{rotation}, then $(h_\lambda, V_\lambda, U_\lambda)$
defined by
\begin{equation*}
\begin{split}
h_\lambda(t,r)=\lambda^2h(\frac{r}{\lambda}, t), \quad U_\lambda(t,r)=\lambda U(\frac{r}{\lambda},t),\quad V_\lambda(t,r)= \lambda V(\frac{r}{\lambda},t)
\end{split}
\end{equation*}
is also a solution of \eqref{rotation}.
Suppose that the solution to \eqref{rotation} enjoys the form \eqref{separation}.
In order to preserve the scaling symmetry of the system, $\mfa_3(r)$ should satisfy
\begin{equation*}
\mfa_3(r)=\lambda \mfa_3(\frac{r}{\lambda}).
\end{equation*}
Hence there exists an $l$ such that $\mfa_3(r)=l r$. Without loss of generality, one may assume $l=1$ so that $V=r\mfv(t)$.

First, we have the following key property for $\mfv(t)$.
\begin{lemma}
Either $\mfv \equiv -\frac{1}{2}$ or 	$\mfv\neq -\frac{1}{2}$.
\end{lemma}
\begin{proof}
Let
\[
\Theta(r,t)=\dfrac{\dfrac{\partial V}{\partial r}+\dfrac{V}{r}+1}{h}\]
  be the relative vorticity of the fluid. One has
		\begin{equation}
	\dfrac{\partial \Theta}{\partial t}+ U \partial_r \Theta=0.
	\end{equation}
	This implies that $\Theta$ is  invariant along the particle path. If $V=r \mfv(t)$, then
\[
\Theta(r,t)=\dfrac{2\mfv+1}{h}.
\]
	Thus $\mfv$ is always $-\frac{1}{2}$ or is never $-\frac{1}{2}$.
	\end{proof}
	
	The case for $\mfv \equiv -\frac{1}{2}$ has been studies in  \cite{CHESN}. It was proved in \cite{CHESN} that if $\mfv\equiv -\frac{1}{2}$, then  the solution of \eqref{rotation}  is either a time periodic solution or a steady solution. The goal of this section is to investigate the case where  $\mfv\neq -\frac{1}{2}$.

If $\mfv\neq -\frac{1}{2}$, then one may assume that
$V=r(e^{g(t)}-\frac{1}{2})$ (or $V=r(-e^{g(t)}-\frac{1}{2})$).
Substituting $V$ into \eqref{rotation} yields
\begin{equation}\label{ODE1}
\left\{
\begin{aligned}
&h=\frac{r^2}{2}\left(e^{2g}-\frac{1}{4}+\frac{1}{2}\frac{d^2g}{dt^2}-\frac{1}{4}\left(\dfrac{dg}{dt}\right)^2 \right),\\
&U=\frac{-r}{2}\frac{dg}{dt},\quad V =r(e^{g(t)}-\frac{1}{2}),
\end{aligned}
\right.
\end{equation}
where $g$ satisfies
\begin{equation}\label{ODE2}
\frac{d^3g}{dt^3}-3\frac{dg}{dt}\frac{d^2g}{dt^2}+\frac{dg}{dt}+\left(\frac{dg}{dt}\right)^3=0.
\end{equation}
The ODE  \eqref{ODE2} can be written as the following ODE system
\begin{equation}\label{ODE3}
\left\{
\begin{aligned}
&\frac{d\xi}{dt}=\eta,\\&
\frac{d\eta}{dt}=\xi(3\eta-\xi^2-1),
\end{aligned}
\right.
\end{equation}
where
\[
\xi=\frac{dg}{dt} \quad \text{and}\quad \eta=\frac{d^2g}{dt^2}.
\]
Define
\begin{equation}\label{deftheta}
\vartheta(t) := \xi^2+1-\eta\quad \text{and}\quad \kappa(t):=\frac{\xi^2+1-2\eta}{(\xi^2+1-\eta)^2} \text{ if } \vartheta\neq 0.
\end{equation}
These two quantities play a crucial role for the study on the ODE system \eqref{ODE3}.

\begin{lemma}\label{lemma 1}
If $\vartheta(0)=0$, then
\begin{equation}\label{explicit1}
\xi=\tan(t+C)\quad \text{and}\quad \eta=\sec^2(t+C).
\end{equation}
If $\vartheta(0)\neq 0$, then  $\kappa(t)\equiv \kappa(0):=\kappa_0\in (-\infty, 1]$ for all $t>0$.
\end{lemma}
\begin{proof}
If $\vartheta(0)=0$, then $\eta=\xi^2+1\neq 0$. Hence $\xi^2+1-2\eta =-(\xi^2+1)\neq 0$.  Furthermore, the straightforward computations yield
\[
\frac{d}{dt}\left(\frac{(\xi^2+1-\eta)^2}{\xi^2+1-2\eta} \right)=0.
\]
Thus one has $\vartheta(t)\equiv 0$. Hence the first equation in \eqref{ODE3} can be written as
\[
\frac{d\xi}{dt} =\xi^2+1.
\]
Hence $\xi =\tan (t+C)$. Using the property $\vartheta\equiv 0$ gives the explicit form of $\eta$ as that in \eqref{explicit1}.

If $\vartheta(0)\neq 0$, then $\kappa(t)\equiv \kappa(0)$ follows from $\kappa'(t)\equiv 0$. Furthermore, the straightforward computations show that  $\kappa_0 \in (-\infty,1]$.
\end{proof}

We have the following result, which is a more detailed version for Theorem \ref{mainthm3}.
\begin{theorem}\label{mainthm3r}
Suppose that the solution $(h, U, V)$ of the system \eqref{rotation} is under the form \eqref{ODE1}, then the following statements hold.
\begin{enumerate}
\item 	If $\kappa_0 \in (0,1]$, the solution of \eqref{ODE3} is periodic  with period $2\pi$. The time  periodic solutions for \eqref{rotation} have periodic particle path.
\item 	 If $\kappa_0\in (-\infty,0]$, the  solutions  of the system \eqref{ODE3} blow up in finite time.
\item  While in the blow up case, every particle path converges to the origin, and along the particle path $h$, $-U$, and $V$ go to $+\infty$ as the time goes to the blowup time. More precisely, there exists a constant $C$ such that the following blowup rates hold,
\begin{equation}\label{blowuprate}
 \frac{C^{-1}}{t_0-t}\leq \sup_{r\in \mathbb{R}_+}|h(t, r)|, \sup_{r\in \mathbb{R}_+}|U(t, r)|^2, \sup_{r\in \mathbb{R}_+}|V(t, r)|^2 \leq  \frac{C}{t_0-t}\quad\text{as }t\rightarrow t_0,
 \end{equation}
 where $t_0$ is the blowup time.
\end{enumerate}
\end{theorem}
\begin{proof} The proof of the theorem is divided into five steps.

{\it Step 1.  Preliminaries.}	 Clearly,  it follows from \eqref{deftheta} that
	\begin{equation}\label{eq4.8}
	\left\{
	\begin{aligned}
	&\xi^2+1=2\vartheta-\kappa_0\vartheta^2\\&
	\eta=\vartheta-\kappa_0\vartheta^2.\\
	\end{aligned}
	\right.
	\end{equation}
	Furthermore, the straightforward computations give
			\begin{equation}\label{ODE4}
	\left\{
	\begin{aligned}	 &\left(\frac{d\vartheta}{dt}\right)^2=\vartheta^2(2\vartheta-\kappa_0 \vartheta^2-1).\\&
	\frac{d\vartheta}{dt}(0)=\xi(0)(\xi^2(0)+1-\eta(0))=\xi(0)\vartheta(0)
	\end{aligned}
	\right.
	\end{equation}

{\it Step 2.  Time periodic solutions.}	
If $\kappa_0\in (0,1]$ and  $\xi(0)\geq 0$, it follows from the first equation in \eqref{eq4.8} that
\[
\vartheta =\frac{1\pm \sqrt{1-(\xi^2+1) \kappa_0}}{\kappa_0}.
\]
Hence $\vartheta\in [\bar\vartheta, \hat\vartheta]$ where
	\begin{equation}\label{defbartheta}
	\bar\vartheta= \frac{1-\sqrt{1-\kappa_0}}{\kappa_0}\quad \text{and}\quad \hat\vartheta= \frac{1+\sqrt{1-\kappa_0}}{\kappa_0}.
	\end{equation}
Furthermore,  there exists a $t'>0$ such that
	\begin{equation}\label{ODE5}
	\frac{d\vartheta}{dt}=\vartheta \sqrt{2\vartheta-\kappa_0 \vartheta^2-1} \quad \text{ for } t\in [0, t').
	\end{equation}

	It follows from the  the standard ODE theory that  there exists  a unique solution to \eqref{ODE5}.
The straightforward calculations yield
	\[
	\int_{\bar\vartheta}^{\hat\vartheta}\frac{1}{\vartheta \sqrt{2\vartheta-\kappa_0 \vartheta^2-1}}d\vartheta=\pi.
	\]
Thus	  $\vartheta$ attains $\hat\vartheta$ in finite time. Furthermore,
\[
\frac{d\vartheta}{dt}=0\quad \text{and}\quad  \frac{d^2\vartheta}{dt^2}=\vartheta(3\vartheta-2\kappa_0\vartheta^2-1)=\vartheta(1-\vartheta)<0\quad \text{at}\,\, \vartheta=\hat{\vartheta}.
 \]
  Similarly, $\vartheta$ attains $\bar\vartheta$ in finite time and
 \[
 \frac{d\vartheta}{dt}=0,
 \quad \frac{d^2\vartheta}{dt^2}=\vartheta(3\vartheta-2\kappa_0\vartheta^2-1)=\vartheta(1-\vartheta)>0\quad \text{at}\,\, \vartheta=\bar{\vartheta}.
 \]
Therefore,  there is a global solution $\vartheta$ for the ODE and $\vartheta$ oscillates between $\bar\vartheta$ and $\hat\vartheta$.  Thus the corresponding solution is periodic with given period.
	The case $\xi(0)<0$ is similar for $0<\kappa_0\leq 1$.
	
The particle path is governed by the following equation
\begin{equation}
\left\{
\begin{aligned}
&\frac{d\tilde{X}(t,x)}{dt}=U(t, \tilde{X}(t,x))=-\dfrac{1}{2}\dfrac{dg}{dt}(t)\tilde{X}(t,x),\\
& \tilde{X}(0,x)=x.
\end{aligned}
\right.
\end{equation}
This yields
\begin{equation}\label{explicit3}
	\tilde{X}(t,x)=e^{\frac{g(0)-g(t)}{2}}x.
\end{equation}
Then the periodic solution corresponding to periodic particle path follows immediately.
		
{\it Step 3. Blowup of the solutions.}  If $\kappa_0=0$, then $\eta=(\xi^2+1)/2$. Hence it follows from the first equation in \eqref{ODE3} that one  has
\[
\frac{d\xi}{dt}=\frac{\xi^2+1}{2}.
\]
This yields that
\begin{equation}\label{explicit2}
\xi=\tan\left(\frac{t+C}{2}\right)\quad \text{and}\quad \eta=\frac{1}{2}\sec^2\left(\frac{t+C}{2}\right).
\end{equation}

If $\kappa_0<0$, $\xi(0)\geq 0$,  and $\vartheta(0)> 0$,
	then one has
	\[
	 \frac{d\vartheta}{dt}=\vartheta \sqrt{2\vartheta-\kappa_0 \vartheta^2-1}>0.
	 \]
	 Thus $\vartheta$ is monotone increasing with respect to $t$. It is easy to prove that $\frac{d\vartheta}{dt}$ is monotone with respect to $t$. Therefore, $\vartheta(t)$ doesn't have an upper bound. Hence one may assume $\vartheta(t)\geq \frac{1}{2}$  for all $t \geq t_*$. Therefore,  $\frac{d\vartheta}{dt}\geq\sqrt{-\kappa_0}\vartheta^2$  for any $t \geq t_*$. Hence the solution blows up in finite time.
	
	If $\vartheta(0)<0$, the one  can similarly prove that  $\vartheta$ attains $\frac{1-\sqrt{1-\kappa_0}}{\kappa_0}$ in finite time and blows up in finite time.
	
  For the case that $\kappa_0<0$, $\xi(0)<0$, by the similar analysis, one can show that $\vartheta$ attains $\frac{-1+\sqrt{1-\kappa_0}}{\sqrt{-\kappa_0}}$ in finite time and goes to $\infty$ in finite time when $\vartheta(0)>0$ or $\vartheta$ attains $\frac{-1-\sqrt{1-\kappa_0}}{\sqrt{-\kappa_0}}$ in finite time and goes to $-\infty$ in finite time when $\vartheta(0)<0$.

{\it Step 4. Blowup quantities.}  If $\kappa_0<0$, then we claim that
 \begin{equation}\label{blowupg}
 \lim_{t\rightarrow t_0}g(t)=+\infty
 \end{equation}
 where $t_0$ is the blowup time.

 If $\kappa_0 =0$ or $\vartheta\equiv 0$, it follows from \eqref{explicit1}  and $\int_{0}^{\frac{\pi}{2}} \tan t dt=+\infty$.

 For $\kappa_0 < 0$, we only present the case for $\xi(0)>0$. The case for $\xi(0)\leq 0$ can be handled similarly.  Without loss of generality, we may assume $\vartheta(0)\geq 1$.
Take
\[
t_0=\int_{\vartheta(0)}^{{\infty}}\frac{1}{s \sqrt{2s-\kappa_0 s^2-1}}ds.
\]
It follows from \eqref{ODE5} that one has
\begin{equation}
\frac{d}{dt}\ln(\vartheta(t))=\sqrt{2\vartheta-\kappa_0\vartheta^2-1}.
\end{equation}
This, together with \eqref{ODE4}, yields
 \begin{equation}\label{eq417}
    g(t)-\ln(\vartheta(t))=g(0)-\ln(\vartheta(0)).
\end{equation}
As $\vartheta(t)$ tends to $+\infty$ when $t \rightarrow t_0$, thus the claim \eqref{blowupg} holds.
By \eqref{ODE1} and \eqref{explicit3}, along the particle path, one has
\begin{equation*}
    U(\tilde{X}(x,t))=-\dfrac{g'(t)x}{2e^{(g(t)-g(0))/2}}.
\end{equation*}
Using L'Hopital's rule gives
\begin{equation}\label{blowupU}
\begin{aligned}
\lim_{t\rightarrow t_0}U(t, \tilde{X}(x,t))=\lim_{t\rightarrow t_0} -\dfrac{g'(t)x}{2e^{(g(t)-g(0))/2}}=\lim_{t\rightarrow t_0}-xe^{\frac{g(0)}{2}}\dfrac{\sqrt{2\vartheta-\kappa_0\vartheta^2-1}}{2e^{\frac{g(0)-\ln(\vartheta(0))}{2}}\sqrt{\vartheta}}=-\infty.
\end{aligned}
\end{equation}
Similarly, one has
\begin{equation}\label{blowupV}
\lim_{t\rightarrow t_0}V(t, \tilde{X}(x,t))=\lim_{g(t)\rightarrow +\infty}xe^{\frac{g(0)-g(t)}{2}}(e^g(t)-\frac{1}{2})=+\infty.
\end{equation}

It follows from \eqref{ODE1} and \eqref{explicit3} that it holds
\begin{equation}\label{blowuph}
\begin{aligned}
    h(t, \tilde{X}(x,t))= &x^2e^{g(0)-g(t)}(e^{2g(t)}-\frac{1}{4}+\eta(t)-\frac{\xi(t)^2}{4})\\
    = &x^{2}e^{g(0)}(e^{g(t)}+\frac{-\kappa_0\vartheta^2(t)}{4e^{g(t)}})\geq x^{2}e^{g(0)+g(t)}.
    \end{aligned}
\end{equation}
Thus one has
$\lim_{t\rightarrow t_0}h(X(x,t))=+\infty$.

{\it Step 5. Blowup rates.}
In fact, the study for the blow up rate for $\vartheta$ can also be regarded as  an alternative proof for the blowup, where the particle path converges to the origin. For simplicity, we assume that $\vartheta(0)>0$ and $\xi(0)>0$.
If $\kappa_0=0$, then it follows from \eqref{explicit2} that one has
\[
\vartheta(t)=\xi^2+1-\eta =\frac{1}{2}\sec^2 \left(\frac{t+C}{2}\right).
\]
Hence $t_0=\pi -C$ and
\[
t_0-t= 2\arcsin \frac{1}{2\vartheta(t)}.
\]
 This shows
\[
\lim_{t\rightarrow t_0}\vartheta(t)(t_0-t)=\lim_{\vartheta\rightarrow \infty}2\vartheta\arcsin{\frac{1}{2\vartheta}}=1.
\]

For $\kappa_0<0$,  one has
\begin{equation}
t_0-t=2\left( \arctan \left(\sqrt{-\frac{\hat{\vartheta}}{\bar{\vartheta}}} \sqrt{\frac{1-\bar{\vartheta}/\vartheta}{1 -\hat{\vartheta}/\vartheta}}\right)  -\arctan\sqrt{-\frac{\hat{\vartheta}}{\bar{\vartheta}}}\right),
\end{equation}
where $\hat\vartheta$ and $\bar\theta$ are defined in \eqref{deftheta}.
Then by L'Hopital's rule, there exists a constant $C$ such that
\[
 \frac{C^{-1}}{t_0-t}\leq \vartheta(t)\leq  \frac{C}{t_0-t}\quad\text{as }t\rightarrow t_0.
 \]
Therefore, this, together with the explicit representations of $h$, $U$, $V$ in terms of $\theta$ in \eqref{eq417}-\eqref{blowuph}, yields \eqref{blowuprate}.

This finishes the proof of the theorem.
\end{proof}

\medskip

{\bf Acknowledgement.}
The research of Xie was partially supported by  NSFC grants 11971307 and 11631008.


\medskip

\begin{thebibliography}{10}

\bibitem{Babin}
A.~Babin, A.~Mahalov, and B.~Nicolaenko, \emph{Global splitting and regularity
  of rotating shallow-water equations}, European J. Mech. B Fluids, \textbf{16}
  (1997),  725--754.

\bibitem{Beale}
A. J. Bourgeois and J. T. Beale, \emph{Validity of the
quasigeostrophic model for large-scale flow in the atmosphere and
ocean}, SIAM J. Math. Anal., \textbf{25} (1994), no. 4, 1023--1068.

\bibitem{Bresch}
D. Bresch, B. Desjardins, and G. M{\'e}tivier, \emph{Recent
  mathematical results and open problems about shallow water equations},
  Analysis and simulation of fluid dynamics, Adv. Math. Fluid Mech.,
  Birkh\"auser, Basel, 2007, 15--31.

\bibitem{CPZ}
G. Chen, R. Pan, and S. Zhu,
\emph{Singularity formation for compressible Euler equations},   SIAM J. Math. Anal., {\bf 49}, (2017), 2591--2614.

\bibitem{ChengMa:zonal} {B. Cheng and A. Mahalov, \emph{Euler equations on a fast rotating sphere -- time-averages and zonal flows}. European J. Mech. - B/Fluids, {\bf 37} (2013), 48-58.}

\bibitem{ChTa:SIAM}
B. Cheng and E. Tadmor, \emph{Long-time existence of smooth solutions for
  the rapidly rotating shallow-water and {E}uler equations}, SIAM J. Math.
  Anal., \textbf{39} (2008),  1668--1685.

\bibitem{ChengXie}
B. Cheng and C. Xie,
\emph{ On the classical solutions of two dimensional inviscid rotating
	shallow water system},  {J. Differential Equations}, \textbf{250} (2011), 690--709.
\bibitem{ChengXie:SIAM}
B. Cheng, P. Qu,  and C. Xie,  \emph{Singularity formation and global existence of classical solutions for one-dimensional rotating shallow water system}, SIAM J.math. Anal., \textbf{50}
(2018),  2486--2508.

\bibitem{CHESN}
A. A. Chesnokov,   \emph{Symmetries and exact solutions of the rotating shallow-water equations} , European Journal of Applied Mathematics, \textbf{20} (2009), 461--477.

\bibitem{Dafermos}
C. Dafermos, {\em Hyperbolic conservation laws in continuum physics}, the third edition, Springer-Verlag, Berlin, 2010.


\bibitem{Delort:1D} J.-M. Delort,  {\em Existence globale et comportement asymptotique pour l'\'equation de Klein-Gordon quasi lin\'eaire \`a donn\'ees petites en dimension 1.} Annales scientifiques de l'Ecole normale sup\'erieure, \textbf{34} (2001), 1--61.

\bibitem{Gill1982} {A. E. Gill,  \emph{Atmosphere-ocean dynamics}, Academic Press, London, 1982.}






\bibitem{Li}
C. Hao, L.  Hsiao, and H.-L. Li, \emph{Cauchy problem
for viscous rotating shallow water equations}, J. Differential
Equations,  \textbf{247} (2009),  3234--3257.


\bibitem{John}
F. John, \emph{Formation of singularities in one-dimensional wave propagation},
Comm. Pure Appl. Math., {\bf 27} (1974), 377--405.


\bibitem{Klainerman}
S. Klainerman, \emph{Global existence of small amplitude solutions to
  nonlinear {K}lein-{G}ordon equations in four space-time dimensions}, Comm.
  Pure Appl. Math., \textbf{38} (1985), 631--641.

\bibitem{Lax}
P. D. Lax, \emph{Development of singularities of solutions of nonlinear hyperbolic partial differential equations}, J. Math. Phys., {\bf 5} (1964), 611--613.

\bibitem{HLiWang}
H. L. Li and Y. Wang, \emph{Formation of singularities of spherically symmetric solutions to the 3D compressible Euler equations and Euler-Poisson equations}, NoDEA Nonlinear Differential Equations Appl., {\bf 25} (2018),  Paper No. 39, 15 pp.

\bibitem{LiWang}
T. Li and D. Wang, \emph{Blowup phenomena of solutions to the Euler equations for compressible fluid flow},  J. Differential Equations, {\bf 221} (2006),  91--101.

\bibitem{LiTa:rotation}
H. Liu and E. Tadmor, \emph{Rotation prevents finite-time breakdown}, Phys. D, \textbf{188}  (2004),  262--276.




\bibitem{Majda}
A. Majda, \emph{Introduction to {PDE}s and waves for the atmosphere and ocean},
  Courant Lecture Notes in Mathematics, vol.~9, New York University Courant
  Institute of Mathematical Sciences, New York, 2003.




\bibitem{Pedlosky}
J.~Pedlosky, \emph{Geophysical fluid dynamics}, Springer Verlag, Berlin, 1992.

\bibitem{Rammaha}
M.~A. Rammaha, \emph{Formation of singularities in compressible fluids in
  two-space dimensions}, Proc. Amer. Math. Soc., \textbf{107} (1989),
  705--714.

\bibitem{Shatah}
J. Shatah, \emph{Normal forms and quadratic nonlinear {K}lein-{G}ordon
  equations}, Comm. Pure Appl. Math., \textbf{38} (1985), 685--696.
\bibitem{Sideris:hyp:singularity}
T.~C. Sideris,  Formation of singularities in solutions to nonlinear hyperbolic equations. Arch. Rational. Mech. Anal.,  \text{86}(1984): 369--381.
\bibitem{Sideris:3D:singularity}
T.~C. Sideris, \emph{Formation of singularities in three-dimensional
  compressible fluids}, Comm. Math. Phys., \textbf{101} (1985), no.~4, 475--485.



\bibitem{SiderisWang}
T. C.  Sideris, B. Thomases, and D.  Wang,
\emph{Long time behavior of solutions to the 3D compressible Euler equations with damping},
Comm. Partial Differential Equations, {\bf 28} (2003),  795--816.





\bibitem{TW}
E. Tadmor and D. Wei,
\emph{On the global regularity of sub-critical Euler-Poisson equations with pressure},
Journal of the European Mathematical Society, \textbf{10} (2008), 757--769.

\end{thebibliography}
\end{document}